\documentclass{siamltex}
\usepackage{amssymb}
\usepackage{bm} 
\usepackage{graphicx}
\usepackage{graphics}
\usepackage{caption2}
\usepackage{psfrag}
\usepackage{amsmath}
\usepackage{amscd}
\usepackage{amsfonts}
\usepackage{float}
\usepackage{latexsym}
\usepackage{lscape}
\usepackage{color}
\newtheorem{remark}{Remark}[section]

\newcommand{\ds}{\displaystyle}
\newcommand{\f}{\frac}

\begin{document}
\title{Computing solution landscape of nonlinear space-fractional problems via fast approximation algorithm}
\author{Bing Yu\thanks{School of Mathematical Sciences, Peking University, Beijing 100871, China (yubing93@pku.edu.cn)}
\and
Lei Zhang\thanks{Beijing International Center for Mathematical Research, Center for Quantitative Biology, Peking University, Beijing 100871, China (zhangl@math.pku.edu.cn)}
\and
Pingwen Zhang\thanks{School of Mathematical Sciences, Laboratory of Mathematics and Applied Mathematics, Peking University, Beijing 100871, China (pzhang@pku.edu.cn).}
\and
Xiangcheng Zheng\thanks{School of Mathematical Sciences, Peking University, Beijing 100871, China (zhengxch@math.pku.edu.cn)}
} 
\maketitle

\begin{abstract}
The nonlinear space-fractional problems often allow multiple stationary solutions, which can be much more complicated than the corresponding integer-order problems. In this paper, we systematically compute the solution landscapes of nonlinear constant/variable-order space-fractional problems. A fast approximation algorithm is developed to deal with the variable-order spectral fractional Laplacian by approximating the variable-indexing Fourier modes, and then combined with saddle dynamics to construct the solution landscape of variable-order space-fractional phase field model. Numerical experiments are performed to substantiate the accuracy and efficiency of fast approximation algorithm and elucidate essential features of the stationary solutions of space-fractional phase field model. Furthermore, we demonstrate that the solution landscapes of spectral fractional Laplacian problems can be reconfigured by varying the diffusion coefficients in the corresponding integer-order problems.
\end{abstract}



\pagestyle{myheadings}
\thispagestyle{plain}

\markboth{Yu, Zhang, Zhang and Zheng}{Solution landscape of space-fractional problems}

\section{Introduction}
Space-fractional partial differential equations (PDEs) provide adequate descriptions for many anomalous processes and complex dynamics, and thus have attracted extensive mathematical and numerical investigations in the past decade \cite{Aco,AntRau,Bon18,Marta,DenJan,Erv,
MaoShe,SheShe,SonXu,TanWan,TiaDu,XuDar,ZayKar}.   For instance, the space-fractional phase field models for two-phase flows were shown to achieve sharper interface compared to the classical integer-order model by simply varying the fractional order \cite{Duo,SonXuKar}. The variable-order space-fractional PDEs, in which the fractional order varies in space to accommodate the influence of variable spatial correlation of medium heterogeneity on tracer dynamics \cite{SunCheChe,Sun14}, improves the modeling capability of constant-order space-fractional PDEs and draws increasing attention in recent years \cite{EliGlu,ZenZha,ZheAA,ZheSINUM,ZhuLiu}.

Apart from focusing on modeling and analysis of nonlinear space-fractional PDEs, how to compute all stationary solutions of the corresponding models, which are the solutions of $\nabla E(u)=0$ for some energy functional $E$, is an important but challenging problem. According to the Morse theory \cite{Milnor}, the local stability of a stationary solution can be characterized by its Morse index, which is defined as the maximal dimension of a subspace on which the Hessian operator is negative definite. Compared with computing the stable (index $=0$) stationary solutions (i.e., minimizers), finding the unstable (index $>0$) stationary solutions (i.e., saddle points) are much more difficult because of their unstable nature.
Extensive numerical algorithms have been developed to compute multiple stationary solutions such as the eigenvector-following method \cite{Doye,ZhaSISC}, the deflation technique \cite{Farr} and the homotopy method \cite{All, Mehta}.
Recently, a solution landscape approach was proposed in \cite{YinPRL} to efficiently find the stationary solutions, including both stable minimizers and unstable saddle points, and their connections. By using the saddle dynamics to compute any-index saddle points \cite{YinSISC}, the downward and upward search algorithms were implemented to construct the solution landscape of the energy systems \cite{YinPRL}, which provides a hierarchy structure that starts with a parent state (the highest-index saddle point) and then relates the low-index saddle points down to the minimizers. 
The solution landscape approach has been successfully applied to study the defect landscape of nematic liquid crystals using a Landau--de Gennes model \cite{HanYin, WZZ}, Ericksen--Leslie Model \cite{Han2021}, and  the molecular Onsager model \cite{Yin2021}.

Despite the aforementioned progresses on finding multiple solutions of nonlinear PDEs, to the best of our knowledge, there is no investigation on solution landscape of the space-fractional problems involving the constant/variable-order fractional Laplacian and the comparisons with the analogous integer-order models in the literature. Due to the nonlocality of the fractional operators and the impact of the variable order, which results in the significant computational consumption, the corresponding fast algorithms are critical for the construction of solution landscape. One difficulty is due to the variable fractional order. Discretization of the variable-order space-fractional operator may generate a dense but non-Toeplitz coefficient matrix. Thus, the efficient computational methods like the fast Fourier transform (FFT) based algorithms cannot be applied directly. In a recent work, Pang and Sun studied a finite difference discretization of a variable-order Riemann-Liouville fractional advection-diffusion equation in one space dimension and approximated the off-diagonal blocks of the resulting stiffness matrices by low-rank matrices \cite{PanSun}. Then a polynomial interpolation based fast algorithm was presented to approximate the coefficient matrices, and the matrix-vector multiplication can be implemented in $O(kN \ln N)$ complexity with $N$ and $k$ being the number of unknowns and the approximant, respectively. Nevertheless, the fast algorithms for the variable-order fractional Laplacian in multiple space dimensions, the definition of which is completely different from the variable-order Riemann-Liouville fractional derivatives, remain untreated. 

Furthermore, a natural but unresolved question lies in the relationship between the variable order and the variable coefficient in fractional models \cite{LorHar,SunCheChe}, that is, is that possible to get identical solutions of variable-order fractional problems by varying the coefficients in their integer-order analogues? This concern is particularly critical for the spectral fractional Laplacian problems as they are defined in terms of the eigenpairs of second-order elliptic operators and may thus exhibit similar features of integer-order or local differential operators.

Motivated by these observations and questions, we investigate the solution landscapes of space-fractional problems involving both constant- and variable-order fractional Laplacian via saddle dynamics, and discover their similarities and differences to analogous integer-order models. The main contributions of the paper are enumerated as follows: 
\begin{itemize}
\item We develop a fast approximation algorithm to deal with the variable-order spectral fractional Laplacian in multiple space dimensions, for which the FFT cannot be applied directly. We approximate the variable-indexing Fourier modes of the fractional Laplacian by the series expansion method and prove that high-order accuracy of the Fourier mode approximation could be obtained using only $O(\ln M)$ terms in the expansion with $M$ being the size of the problems. We accordingly prove that computing the corresponding approximation of the variable-order spectral fractional Laplacian only requires $O(M\ln^2 M)$ operations by employing the FFT, which is much faster than implementing the direct discretization schemes of variable-order spectral fractional Laplacian that need $O(M^2\ln M)$ operations. We further implement this approximation scheme to saddle dynamics in order to compute any-index saddle points in the solution landscape.

\item We systematically construct the solution landscapes for both constant- and variable-order space-fractional phase field models, and compare them with their integer-order counterparts. Numerical experiments show that new stationary solutions emerge when the constant fractional order decreases, resulting in more complex solution landscape than their integer-order analogues. Moreover, due to the impact of variable fractional order, the solution landscape of variable-order fractional model becomes more complicated than the constant-order fractional model. We further demonstrate that the solution landscapes of spectral fractional Laplacian problems can be reconfigured by varying the diffusion coefficients in corresponding integer-order problems, which indicates the locality of spectral fractional Laplacian  as integer-order differential operators.
\end{itemize}

The rest of the paper is organized as follows: In Section 2 we go over preliminaries on fractional Laplacian operators and phase field models. In Section 3 we develop and analyze a fast approximation to the variable-order fractional Laplacian and carry out numerical experiments to test its accuracy and efficiency. In Section 4 we apply the fast approximation to saddle dynamics in order to construct the solution landscape of space-fractional phase field models. In Section 5 numerical experiments are performed to investigate the solution landscapes of space-fractional phase field model and make comparisons with their integer-order analogues.

\section{Preliminaries}
In this section we present brief reviews for the constant-order and variable-order fractional Laplacian operators and the corresponding phase field model.
\subsection{FFT-based approximation to fractional Laplacian}
Let $\mathbb Z$ be the set of integers, $N>0$ be an even number, $0\leq m\in\mathbb Z$, $\mathbb N:=\{z\in \mathbb Z:-N/2\leq z\leq N/2-1\}$, $\Omega=[0,1]^2$, $L^2(\Omega)$ be the space of square-integrable functions on $\Omega$ with the norm $\|\cdot\|_{L^2(\Omega)}$, $L^2_{per}(\Omega)$ be the space of $L^2$ functions with periodic boundary conditions on $\Omega$, and $H^m(\Omega)$ be the space of $L^2$ functions with $L^2$ derivatives up to order $m$ \cite{AdaFou}. A function $u\in L^2_{per}(\Omega)$ may be expanded as 
$$u(x,y)=\sum_{k,l\in\mathbb Z}u_{k,l}e^{2\pi \mathbf i(kx+ly)}$$
where $\bf i$ stands for the imaginary unit and $\{u_{k,l}\}$ are the Fourier coefficients of $u$ given as
$$u_{k,l}=\int_{0}^1\int_0^1u(x,y)e^{-2\pi \mathbf i(kx+ly)}dydx. $$
Then the fractional Laplacian $(-\Delta)^{\alpha/2}u(x,y)$ of order $0<\alpha\leq 2$ is defined for any $u\in L^2_{per}(\Omega) $ by \cite{AinMao,ZhaJia}
\begin{equation}\label{FL}
(-\Delta)^{\alpha/2}u(x,y)=\sum_{k,l\in\mathbb Z}\big[4\pi^2(k^2+l^2)\big]^{\alpha/2}u_{k,l}e^{2\pi \mathbf i(kx+ly)}.
\end{equation}
In practical computations, we only evaluate the truncation $(-\Delta)^{\alpha/2}_Nu$ of $(-\Delta)^{\alpha/2}u$ defined by
\begin{equation}\label{TrFL}
(-\Delta)^{\alpha/2}_Nu(x,y):=\sum_{k,l\in \mathbb N}\big[4\pi^2(k^2+l^2)\big]^{\alpha/2}u_{k,l}e^{2\pi \mathbf i(kx+ly)}.
\end{equation}

Define a uniform partition of $\Omega$ by $0=x_0<x_1<\cdots<x_N=1$ and $0=y_0<y_1<\cdots<y_N=1$ with $x_k=kh$ and $y_l=lh$ for $0\leq k,l\leq N$ and $h=1/N$. Then the Fourier coefficients $\{u_{k,l}\}_{k,l\in\mathbb N}$ could be approximated by
\begin{equation}\label{dft}
u_{k,l}\approx \hat u_{k,l}:=h^2\sum_{0\leq i,j< N}u(x_i,y_j)e^{-2\pi \mathbf i(kx_i+ly_j)}.
\end{equation}  
Then we replace $u_{k,l}$ by $\hat u_{k,l}$ in (\ref{TrFL}) to find a further approximation $(-\Delta)^{\alpha/2}_{N,h}u$ of $(-\Delta)^{\alpha/2}u$ given by
\begin{equation}\label{idft}
(-\Delta)^{\alpha/2}_{N,h}u(x_i,y_j):=\sum_{k,l\in \mathbb N}\big[4\pi^2(k^2+l^2)\big]^{\alpha/2}\hat u_{k,l}e^{2\pi \mathbf i(kx_i+ly_j)},~~0\leq i,j<N.
\end{equation}
The estimates of these discretizations are standard and we refer \cite{SheTanWan} for details. It is known that either the discrete Fourier transform (\ref{dft}) or the inversion procedure (\ref{idft}) could be efficiently carried out by means of FFT, which significantly reduces the computational consumption compared with the direct method.
\subsection{Variable-order fractional Laplacian}
Variable-order fractional Laplacian operator of order $0< \alpha(x,y)\leq 2$ is defined as \cite{AntCer,Far,Sun14,ZenZha,ZhuLiu}
\begin{equation}\label{VoFL}
(-\Delta)^{\alpha(x,y)/2}u(x,y):=\sum_{k,l\in\mathbb Z}\big[4\pi^2(k^2+l^2)\big]^{\alpha(x,y)/2}u_{k,l}e^{2\pi \mathbf i(kx+ly)}.
\end{equation}
Here we assume that $\alpha\in L^2_{per}(\Omega)$. 
\begin{remark}
The definition at $(x,y)=(\tilde x,\tilde y)$ may be rewritten as
 \begin{equation}\label{Direct}
(-\Delta)^{\alpha(\tilde x,\tilde y)/2}u(\tilde x,\tilde y)=\Big(\sum_{k,l\in\mathbb Z}\big[4\pi^2(k^2+l^2)\big]^{\alpha(\tilde x,\tilde y)/2}u_{k,l}e^{2\pi \mathbf i(kx+ly)}\Big)\Big|_{(x,y)=(\tilde x,\tilde y)},
\end{equation}
that is, it is equivalent to the constant-order fractional Laplacian (\ref{FL}) with $\alpha=\alpha(\tilde x,\tilde y)$ evaluated at $(x,y)=(\tilde x,\tilde y)$.
\end{remark}

We could follow (\ref{FL})-(\ref{idft}) to discretize $(-\Delta)^{\alpha(x,y)/2}u$ as
\begin{equation}\label{VOidft}
(-\Delta)^{\alpha(x_i,y_j)/2}_{N,h}u(x_i,y_j)=\sum_{k,l\in \mathbb N}\big[4\pi^2(k^2+l^2)\big]^{\alpha(x_i,y_j)/2}\hat u_{k,l}e^{2\pi \mathbf i(kx_i+ly_j)}
\end{equation}
for $0\leq i,j<N$. However, due to the variability of the fractional order, we could no longer apply the FFT for the efficient implementation, while the direct method is computationally expensive. To circumvent this difficulty, we develop a fast approximation for the variable-order fractional Laplacian operator based on the series expansion of the variable Fourier multiplier $\big[4\pi^2(k^2+l^2)\big]^{\alpha(x,y)/2}$ in Section \ref{sec:fa}.

\subsection{Space-fractional phase field model}
Phase field models have been developed for modeling the microstructure evolution precesses and interfacial problems such as grain growth, dislocation dynamics and phase transformation \cite{ChenLQ, ZhangChe}. For a mixture of two fluids, an order parameter $-1\leq u\leq 1$ is introduced to identify the regions occupied by the two fluids. A simple form of the Ginzburg-Laudau free energy with a double-well potential is given by 
\begin{equation}\label{PH}
	E(u) = \int_{\Omega} \frac{\kappa}{2} |\nabla u(x)|^2 + \frac{1}{4}\big(1-u^2(x)\big)^2 dx, 
\end{equation}
and the related phase field (Allen-Cahn) equation is 
\begin{equation}\label{AC}
	\dot{u} = -\frac{\delta E}{\delta u} = \kappa \Delta u + u - u^3, 
\end{equation}
where $\kappa$ is the diffusion coefficient,  a critical parameter determining the complexity of the system. 

In this paper, we consider the following variable-order space-fractional phase field equation
\begin{equation}\label{mh}
	\dot{u}=F(u):=-\kappa (-\Delta)^{\alpha(x,y)/2}u+u-u^3
\end{equation}
equipped with the periodic boundary condition, in which the Laplace operator $-\Delta$ in the classical phase field model \eqref{AC} is replaced by the variable-order space-fractional Laplace operator \eqref{VoFL} to accommodate the heterogeneity of the surroundings. 
It is worth mentioning that when the order $\alpha$ is not a constant, the dynamical system \eqref{mh} is a non-gradient system without the energy functional $E(u)$.

\section{Fast approximation to variable-order fractional Laplacian}\label{sec:fa}
We propose and analyze a fast scheme to the variable-order fractional Laplacian (\ref{VoFL}), and test its accuracy and efficiency in this section.
\subsection{One-dimensional illustration}\label{sec:1D}
We propose the key idea of the fast approximation by the one-dimensional (1D) analogue of (\ref{VoFL})
\begin{equation}\label{VoFL1D}
(-\Delta)^{\alpha(x)/2}u(x)=\sum_{k\in\mathbb Z}\big(4\pi^2k^2\big)^{\alpha(x)/2}u_{k}e^{2\pi \mathbf i kx}
\end{equation}
where the Fourier coefficients $\{u_k\}$ and their discretizations $\{\hat u_k\}$ are given by
$$u_{k}=\int_{0}^1u(x)e^{-2\pi \mathbf i kx}dx,~~\hat u_{k}=h\sum_{i=0}^{N-1}u(x_i)e^{-2\pi \mathbf i kx_i}. $$
Then we follow (\ref{FL})-(\ref{idft}) to discretize $(-\Delta)^{\alpha(x)/2}u$ as
\begin{equation}\label{VoFL1DTru}
(-\Delta)_{N,h}^{\alpha(x_i)/2}u(x_i)=\sum_{k\in \mathbb N}\big(4\pi^2k^2\big)^{\alpha(x_i)/2}\hat u_{k}e^{2\pi \mathbf i kx_i},~~0\leq i<N.
\end{equation}
We treat $\big(4\pi^2k^2\big)^{\alpha(x)/2}$ as a power function $g_k(z):=\big(4\pi^2k^2\big)^{z/2}$ for $k\neq 0$ and $0<z\leq 2$ such that $g_k(\alpha(x))=\big(4\pi^2k^2\big)^{\alpha(x)/2}$, and expand $g_k$ at $z=1$ as follows
\begin{equation}\label{gk}
\begin{array}{rl}
\ds  g_k(z)\hspace{-0.1in}&\ds=\sum_{s=0}^S \frac{g^{(s)}_k(1)}{\Gamma(s+1)}(z-1)^s+\frac{g^{(S+1)}_k(\xi)}{\Gamma(S+2)}(z-1)^{S+1}\\[0.1in]
&\ds=\sum_{s=0}^S \frac{\big(4\pi^2k^2\big)^{1/2}}{\Gamma(s+1)}\frac{\ln^s (4\pi^2 k^2)}{2^s}(z-1)^s\\[0.1in]
&\ds\qquad+\frac{\big(4\pi^2k^2\big)^{\xi/2}}{\Gamma(S+2)}\frac{\ln^{S+1} (4\pi^2 k^2)}{2^{S+1}}(z-1)^{S+1}=:G_{k}(z)+R_k(z).
\end{array} 
\end{equation}
Here $\xi$ lies in between $1$ and $z$ and $\Gamma(\cdot)$ represents the Gamma function. We first analyze the truncation error $R_k(z)$ in the following theorem.
\begin{theorem}\label{thm:R}
For $0<m\in \mathbb Z$, let $S=\big\lceil e^{\mu+1} \ln (\pi N)-1\big\rceil  $ with $\mu$ satisfying
\begin{equation}\label{mucond}
\mu e^{\mu+1}\geq m+2.
\end{equation}
Then the truncation error in (\ref{gk}) can be bounded by 
$$|R_k(z)|\leq N^{-m},~~k\in\mathbb N/\{0\},~~0<z\leq 2.$$
\end{theorem}
\begin{remark}
This theorem implies that by selecting $S=O(\ln N)$, the truncation error $R_k(z)$ could be sufficiently small with high-order decay. 
\end{remark}
\begin{proof}
 We apply $|k|\leq N/2$, $|z-1|\leq 1$ and the Stirling's formula 
 $$\Gamma(S+2)>\frac{(S+1)^{S+3/2}}{e^{S+1}}$$ to bound the truncation error as follows
\begin{equation}\label{gky}
\begin{array}{rl}
\ds |R_k(z)|\hspace{-0.1in}&\ds =\bigg|\frac{\big(4\pi^2k^2\big)^{\xi/2}}{\Gamma(S+2)}\frac{\ln^{S+1} (4\pi^2 k^2)}{2^{S+1}}(z-1)^{S+1}\bigg|\\
&\ds\leq \frac{4\pi^2k^2}{(S+1)^{S+3/2}e^{-(S+1)}}\frac{\ln^{S+1} (4\pi^2 k^2)}{2^{S+1}}\\
&\ds\leq 4\pi^2k^2\bigg(\frac{e\ln (4\pi^2 k^2)}{2(S+1)}\bigg)^{S+1}\leq \pi^2N^2\bigg(\frac{e\ln (\pi^2 N^2)}{2(S+1)}\bigg)^{S+1}.
\end{array}
\end{equation}  
Therefore, $S$ may be chosen such that the following inequality holds 
\begin{equation}\label{gky1}
 \pi^2N^2\bigg(\frac{e\ln (\pi^2 N^2)}{2(S+1)}\bigg)^{S+1}\leq N^{-m},
\end{equation}
which is equivalent to 
\begin{equation}\label{gky2}
\begin{array}{rl}
\ds \ln \pi^2+(S+1)\ln\frac{e\ln (\pi^2 N^2)}{2(S+1)}\hspace{-0.1in}&\ds\leq -(m+2)\ln N\\[0.15in]
&\ds=-\frac{m+2}{2}\ln (\pi^2N^2)+\frac{m+2}{2}\ln (\pi^2).
\end{array}
\end{equation}
Thus it suffices to find $S$ such that
\begin{equation}\label{gky3}
(S+1)\ln\frac{e\ln (\pi^2 N^2)}{2(S+1)}\leq -\frac{m+2}{2}\ln( \pi^2N^2).
\end{equation}
To further simplify this equation, we introduce the substitution 
\begin{equation}\label{sub}
S+1=\frac{e^{\mu+1}}{2} \ln (\pi^2N^2)\mbox{ for }S>\frac{e^{2}}{2} \ln (\pi^2N^2)-1
\end{equation} 
and for some $\mu>1$. Invoking this transformation in (\ref{gky3}) yields
\begin{equation}\label{gky4}
\ds (S+1)(-\mu)\leq -\frac{m+2}{2}\ln( \pi^2N^2).
\end{equation}
Combining (\ref{sub}) and (\ref{gky4}) leads to
\begin{equation}
\frac{e^{\mu+1}}{2} \ln (\pi^2N^2)\geq\frac{m+2}{2\mu}\ln( \pi^2N^2),
\end{equation}
that is, $\mu e^{\mu+1}\geq m+2$, which is exactly (\ref{mucond}).
We thus conclude that by selecting 
$$S=\bigg\lceil \frac{e^{\mu+1}}{2} \ln (\pi^2N^2)-1\bigg\rceil  $$ with $\mu$ satisfying (\ref{mucond}), the truncation error in (\ref{gk}) can be bounded by $|R_k(z)|\leq N^{-m}$, which completes the proof.
\end{proof}

Based on this theorem, we could substitute $g_k(\alpha(x))$ by $G_k(\alpha(x))$ for $k\in\mathbb N/\{0\}$ in (\ref{VoFL1DTru}) and notice that $g_0(\alpha(x))=0$ to reach a further approximation 
 \begin{equation}\label{VoFL1DApp}
 \begin{array}{l}
 \ds(-\Delta)_{N,h,F}^{\alpha(x_i)/2}u(x_i)= \sum_{k\in \mathbb N/\{0\}}G_k(\alpha(x_i))\hat u_{k}e^{2\pi \mathbf ikx_i} \\[0.1in]
 \ds\quad\quad\quad=\sum_{s=0}^S(\alpha(x_i)-1)^s\sum_{k\in \mathbb N/\{0\}} \frac{\big(4\pi^2k^2\big)^{1/2}}{\Gamma(s+1)}\frac{\ln^s (4\pi^2 k^2)}{2^s} \hat u_{k}e^{2\pi \mathbf ikx_i}
 \end{array}
\end{equation}
for $0\leq i<N$. 
\begin{theorem}\label{thm:cpu}
The implementation of (\ref{VoFL1DApp}) for $0\leq i<N$ requires $O(N\ln^2 N)$ operations via the FFT, which is much faster than the evaluation of (\ref{VoFL1DTru}) for $0\leq i<N$ that needs $O(N^2\ln N)$ operations.
\end{theorem} 
\begin{proof}
For the direct implementation by (\ref{VoFL1DTru}), although $\{\hat u_k\}_{k=-N/2}^{N/2-1}$ could be efficiently computed via the FFT and stored, at each $x_i$ all Fourier modes in (\ref{VoFL1DTru}) should be updated and the inverse FFT needs to be performed due to the varying of the fractional order $\alpha(x)$. Therefore, totally $N\ln N+N\cdot O(N\ln N)=O(N^2\ln N)$ operators are required, which proves the computational cost of (\ref{VoFL1DTru}).

We then consider the computational cost of the fast scheme (\ref{VoFL1DApp}). By Theorem \ref{thm:R}, $S$ has only the magnitude of $O(\ln N)$ and for each $0\leq s\leq S$, the inner summation with $0\leq i<N$ can be evaluated efficiently via the FFT and inverse FFT, which requires $O(N\ln N)$ operators. Therefore, the total conputaitonal cost of the fast scheme (\ref{VoFL1DApp}) is $O(N\ln^2N)$, which completes the proof.
\end{proof}
 \begin{remark} 
 If we directly evaluate the sum on the right-hand side of (\ref{VoFL1DTru}) for each $i$ without using the inverse Fourier transform, the total computational cost of  (\ref{VoFL1DTru}) for $0\leq i<N$ is $O(N^2)$, which is a bit smaller than $O(N^2\ln N)$ proved in the above theorem. However, this could introduce numerical inaccuracies and inconveniences compared with employing the inverse Fourier transform. Furthermore, improving the computational efficiency by a logarithmic factor $\ln N$ does not bring essential differences.
 \end{remark}
\subsection{Accuracy and efficiency tests}
We numerically test both the accuracy and the computational consumptions (i.e., the CPU times) of the fast method (\ref{VoFL1DApp}) by comparing it with the direct method in both constant-order and variable-order cases. Here the direct method for constant-order fractional Laplacian means evaluating the 1D analogue of (\ref{idft}) by the FFT and its inversion, and for the variable-order case, the direct method means evaluating it pointwisely like a constant-order fractional Laplacian. 

For the accuracy test, we choose $u(x)=x^2(1-x)^2$ or $u(x)=\sin(6\pi x)$ under $N=2^{12}$ or $N=2^{18}$ and $\alpha(x)=1.5$ or $\alpha(x)=1.5+0.4\sin(2\pi x)$, and the $L^2$ errors between the fast method and the direct method are presented in Figure \ref{plot1}. We find that the machine accuracy was reached in all three plots when $S> 25$, and a significant increment of $N$ (e.g., $N$ becomes $2^{18}$ from $2^{12}$ in  Figure \ref{plot1}(B)) only has slight affects on the threshold value of $S$ (that is, $S$ should be bigger than $25$ in Figure \ref{plot1}(B) instead of $22$ in Figure \ref{plot1}(A) such that the the machine accuracy could be reached). These observations are consistent with the analysis in Theorem \ref{thm:R} in that the $S$ does not heavily depend on $N$ and demonstrate the effectiveness of the fast method (\ref{VoFL1DApp}).

\begin{figure}[htbp]
	\setlength{\abovecaptionskip}{0pt}
	\centering		
	\includegraphics[width=1\linewidth]{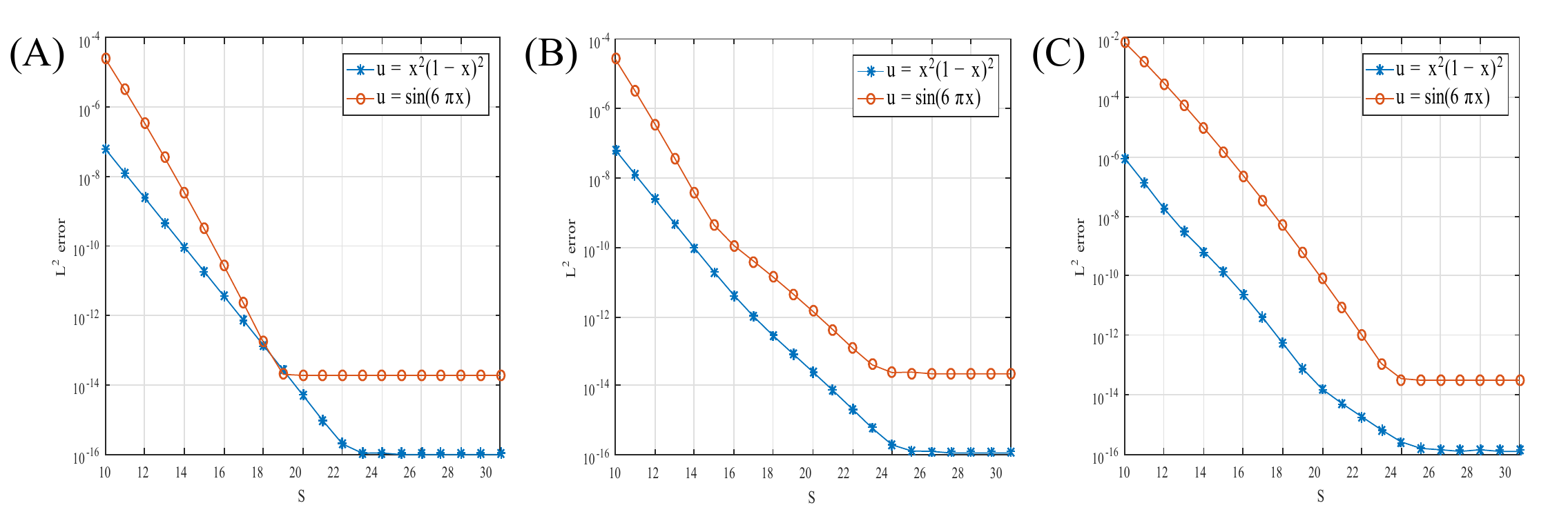}
	\caption{Plots of $L^2$ errors under (A) $\alpha=1.5$ and $N=2^{12}$; (B) $\alpha=1.5$ and $N=2^{18}$ and (C) $\alpha=1.5+0.4\sin(2\pi x)$ and $N=2^{12}$.}
	\label{plot1}
\end{figure}

We then test the efficiency of the fast method. Motivated by the observations, we take $S=25$ and choose $u(x)=x^2(1-x)^2$ and $\alpha(x)=1.5$ or $\alpha(x)=1.5+0.4\sin(2\pi x)$, and the results are presented in Figure \ref{plot2}. In Figure \ref{plot2}(A), in which the $\alpha$ is a constant, the CPU times of both the direct method and the fast method grow linearly as their operations are $O(N\ln N)$ and $O(N\ln^2 N)$, respectively. In particular, the CPU times of the fast method are roughly $S$ multiples of the corresponding CPU times of the direct method due to the series expansion. However, in the case of variable fractional order in Figure \ref{plot2}(B), the CPU times of the direct method grow quadratically while the fast method still keeps a linear growth of the CPU time, as their operations are $O(N\ln^2N)$ and $O(N^2\ln N)$, respectively, as proved in Theorem \ref{thm:cpu}. These tests fully substantiate the efficiency of the fast method.

\begin{figure}[htbp]
	\setlength{\abovecaptionskip}{0pt}
	\centering
	\includegraphics[width=1\linewidth]{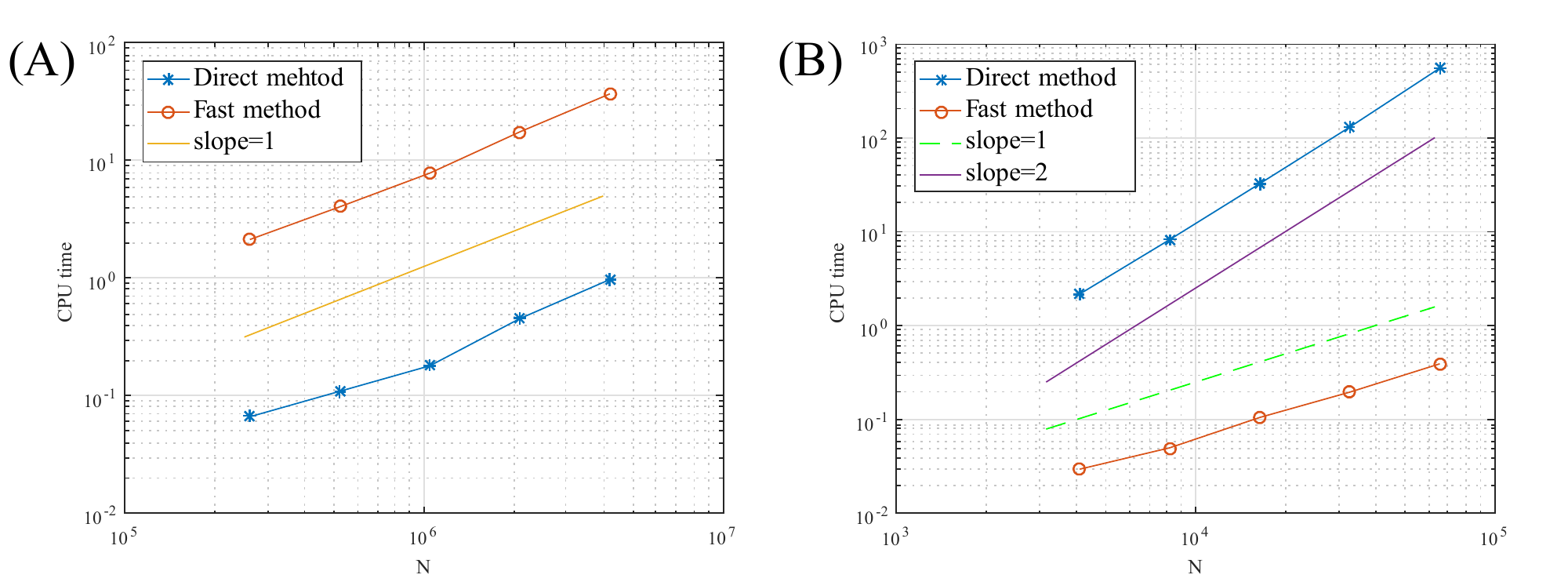}
	\caption{Plots of CPU times under (A) $\alpha=1.5$ and (B) $\alpha=1.5+0.4\sin(2\pi x)$.}
	\label{plot2}
\end{figure}

\subsection{Two-dimensional (2D) extension}
For the 2D approximated operator (\ref{VOidft}), we could similarly consider the mode $\big(4\pi^2(k^2+l^2)\big)^{\alpha(x,y)/2}$ as a power function $g_{k,l}(z):=\big(4\pi^2(k^2+l^2)\big)^{z/2}$ for $(k,l)\neq (0,0)$ and expand this function at $z=1$ like (\ref{gk}) to find the following approximation to $(-\Delta)_{N,h}^{\alpha(x,y)/2}u(x,y)$
 \begin{equation}\label{VoFL1DApp22}
 \begin{array}{l}
 \ds (-\Delta)_{N,h,F}^{\alpha(x_i,y_j)/2}u(x_i,y_j):=\sum_{s=0}^S(\alpha(x_i,y_j)-1)^s\\[0.1in]
 \ds\qquad\times\sum_{k,l\in\mathbb N,(k,l)\neq(0,0)} \frac{\big(4\pi^2(k^2+l^2)\big)^{1/2}}{\Gamma(s+1)}\frac{\ln^s (4\pi^2 (k^2+l^2))}{2^s} \hat u_{kl}e^{2\pi \mathbf i(kx_i+ly_j)}
 \end{array}
\end{equation}
for $0\leq i,j<N$, where, similar to the derivations in Section \ref{sec:1D}, we require
$$S=\big\lceil e^{\mu+1} \ln (\sqrt{2}\pi N)-1\big\rceil  $$ with $\mu$ satisfying (\ref{mucond}) in order that the error of approximating $\big(4\pi^2(k^2+l^2)\big)^{\alpha(x,y)/2}$ can be bounded by 
$ N^{-m}$.

It is worth mentioning that the higher dimension of the problem is, the more computational cost we could save. This is due to the fact that the number of performing FFT and its inversion could be reduced from $M$ in the direct method to $S=O(\ln M)$ in the fast method where $M$ represents the size of the problem. Therefore, the high dimension will lead to a very large $M$ and consequently the fast method will significantly reduce the computational costs.

\section{Construction of solution landscape}\label{sec:sd}
In this section we combine the saddle dynamics with fast approximation algorithm of the variable-order fractional Laplacian to construct the solution landscapes of the variable-order space-fractional problem. 
\subsection{Saddle dynamics}\label{sec:SD}
The saddle dynamics is designed for searching the saddle points of a dynamical system $
	\dot{\bm u} = \bm F(\bm u),
$
where $\bm F:\mathbb R^M\rightarrow \mathbb R^M$ for $0<M\in\mathbb Z$ is a twice continuously differentiable function and $\bm u\in \mathbb R^M$ \cite{Wig}. 
For a gradient system with an energy function $E(\bm u)$, $\bm F$ equals to the negative gradient of $E(u)$, and the Morse index of a stationary solution
 $\bm u^*$ is equal to the number of negative eigenvalues of the Hessian $G(\bm u^*) = \nabla^2 E(\bm u^*)$. Denote the eigenvectors corresponding to negative eigenvalues as $\{\bm v_i\}_{i=1}^k$ and define the unstable subspace as $\mathcal{W}^u:=\text{span}\{\bm v_1, \cdots, \bm v_k\}$. Then $\bm u^*$ is a local maximum on the linear manifold $\bm u^* + \mathcal{W}^u$ and a local minimum on $\bm u^* + (\mathcal{W}^u)^\perp$. 
Based on this property, the saddle dynamics for an index-$k$ saddle point of the gradient system \cite{YinSISC} is written as
\begin{equation}\label{HiSD}
	\left\{
	\begin{array}{l}
		\ds \dot{\bm u} =\Big(I -2\sum_{j=1}^k \bm v_j \bm v_j^\top \Big)\bm F(\bm u),\\[0.025in]
		\ds \dot{\bm v}_i = -\Big(I-\bm v_i \bm v_i^\top - 2\sum_{j=1}^{i-1} \bm v_j \bm v_j^\top \Big)G(\bm u)\bm v_i,~~ i = 1, \cdots, k, \\[0.175in]
	\end{array}
	\right.
\end{equation}
with initial conditions
\begin{equation}\label{IC}
	\ds \bm u(0)=\bm u^0 \in \mathbb R^M, ~~\bm v_i(0)=\bm v_i^0 \in \mathbb R^M ,~~\bm v_i^\top \bm v_j=\delta_{ij},~~  i,j = 1, \cdots, k. 
\end{equation}
Here $\bm u$ is a state variable and $\{\bm v_i\}_{i = 1}^k$ are $k$ direction variables that are used to approximate an orthogonal basis of the unstable subspace $\mathcal{W}(\bm u)$.
The dynamics of $\bm u$ in \eqref{HiSD} represents a transformed gradient flow
\begin{equation}\label{GF}
	\dot{\bm u} = -\mathcal{P}_{\mathcal{W}^u(\bm u)} \bm F (\bm u) + \mathcal{P}_{(\mathcal{W}^u(\bm u))^\perp} \bm F (\bm u) = (I - 2 \mathcal{P}_{\mathcal{W}^u(\bm u)}) \bm F (\bm u), 
\end{equation}
while the dynamics of $\bm v_i$ in \eqref{HiSD} is to find the unit eigenvector of the $i$-th smallest eigenvalue of $G(\bm u)$ with the knowledge of $\bm v_1, \cdots, \bm v_{i-1}$.

For a general non-gradient system, the Jacobian $J(\bm u) = \nabla \bm F(\bm u)$ should be used to substitute the self-joint Hessian in a gradient system. The index of a stationary point should be defined as the number of eigenvalues with positive real parts. The corresponding saddle dynamics for the non-gradient system \cite{YinSCM} reads 
\begin{equation}\label{GHiSD}
	\left\{
	\begin{array}{l}
		\ds \dot{\bm u} =\Big(I -2\sum_{j=1}^k \bm v_j \bm v_j^\top \Big)\bm F(\bm u),\\[0.025in]
		\ds \dot{\bm v}_i= \big(I-\bm v_i\bm v_i^\top\big)J(\bm u)\bm v_i-\sum_{j=1}^{i-1} \bm v_j \bm v_j^\top \big(J(\bm u)+J^\top(\bm u)\big) \bm v_i,~~ i = 1, \cdots, k. \\[0.175in]
	\end{array}
	\right.
\end{equation}
It was shown in \cite[Theorem 2.1]{YinSCM} that a linearly stable solution of (\ref{GHiSD}) is an index-k stationary point of the system $
\dot{\bm u} = \bm F(\bm u)$. 

The numerical discretizations for \eqref{HiSD} and \eqref{GHiSD} could have a uniform form as follows 
\begin{equation}\label{DGHiSD}
\left\{
\begin{array}{l}
\ds \bm u^{m+1} = \bm u^{m}+\tau\Big(I -2\sum_{j=1}^k \bm v^{m}_j(\bm v^{m}_j)^\top \Big)\bm F(\bm u^{m}),\\[0.025in]
\ds \tilde{\bm v}^{m+1}_i=\tilde{\bm v}^{m}_i+\sigma \frac{\bm F(\bm u^{m+1}+l \bm v_i^{m})-\bm F(\bm u^{m+1}-l \bm v_i^{m})}{2l},~~  i = 1, \cdots, k,\\[0.175in]
\ds \big[\bm v^{m+1}_1,\cdots,\bm v^{m+1}_k\big]=\mbox{orth}\big[\tilde{\bm v}^{m+1}_1,\cdots,\tilde{\bm v}^{m+1}_k\big].
\end{array}
\right.
\end{equation}
Here $l > 0$ is a small constant and $\tau,\sigma>0$ are step sizes. A dimer centered at $\bm u^{m + 1}$ with the direction $\bm v_i$ and length $2l$
\begin{equation}\label{dimer}
	\f{\bm F(\bm u^{m+1}+l \bm v_i^{m}) - \bm F(\bm u^{m+1}-l \bm v_i^{m})}{2l}
\end{equation}
is applied to approximate $J(\bm u^{m + 1})\bm v_i$, and $\text{orth}(\cdot)$ represents the normalized orthogonalization function, which could be realized by a Gram-Schmidt procedure.

When $\bm F$ involves the variable-order fractional Laplacian (\ref{VoFL}), we can implement the fast evaluation schemes (\ref{VoFL1DApp}) or (\ref{VoFL1DApp22}) in the numerical discretization of saddle dynamics (\ref{DGHiSD}). For instance, in the space-fractional phase field model \eqref{PH}, we could formulate the corresponding saddle dynamics and apply \eqref{VoFL1DApp22} to evaluate
$$\begin{array}{l}
	\ds F(u)|_{(x,y)=(x_i,y_j)}\approx \hat F_{i,j}\\[0.05in]
	\ds\qquad:=-\kappa (-\Delta)_{N,h,F}^{\alpha(x_i,y_j)/2}u(x_i,y_j)+u(x_i,y_j)-u^3(x_i,y_j),~~0\leq i,j<N.
\end{array}  $$

\subsection{Algorithm of solution landscape}


Using the saddle dynamics in Section \ref{sec:SD}, solution landscape of the aforementioned dynamical system can be constructed systematically via the downward and upward search algorithms \cite{YinSCM}. 

The downward search algorithm enables us to search lower-index saddles and stable states from a high-index saddle point. Given an index-$k$ saddle point and its $k$ unstable directions denoted by $(\hat{\bm u}, \hat{\bm v}_1, \cdots, \hat{\bm v}_k)$ as a parent state, we could choose index-$m$ saddles with $0 \le m < k$ as our targets. In practice, a perturbed $\hat{\bm u}$ along the direction $\hat{\bm v}_i $ for some $m < i \le k$ with the directions $\{\hat{\bm v}_j\}_{j=1}^m$, that is, $(\hat{\bm u} \pm \epsilon \hat{\bm v}_i, \hat{\bm v}_1, \cdots, \hat{\bm v}_{m})$, will serve as the initial conditions in (\ref{IC}) in the implementation of the discrete saddle dynamics (\ref{DGHiSD}). By repeating this procedure on the newly found saddles, the pathway map under the given index-$k$ saddle could be constructed completely.

Inversely, the upward searching is adopted when the parent state is unknown or multiple parent states exist. 
To search an index-$m$ saddle point ($m > k$) from an index-$k$ saddle point $(\hat{\bm u}, \hat{\bm v}_1, \cdots, \hat{\bm v}_k)$, 
we need to compute $m$ eigendirections by the dynamics of $v_i$ in \eqref{DGHiSD} as ascent searching directions, where the first $k$ directions have been given by the index-$k$ saddle. Then we could start from the initial states $(\hat{\bm u} \pm \epsilon \hat{\bm v}_i, \hat{\bm v}_1, \cdots, \hat{\bm v}_{m})$ for some $k <i \le m$ to search the index-$m$ saddle point. Combining the downward and upward search algorithms, we could navigate up and down on the solution landscape and find all connected saddle points and stable states. 

\section{Numerical experiments}
We perform numerical simulations to construct the solution landscapes of constant-order and variable-order space-fractional phase field models (\ref{AC}) in one and two space dimensions and compare them with their integer-order analogues.
\subsection{One-dimensional problem}
We first consider 1D space-fractional phase field equation (\ref{mh}) to understand the global structures of this model. The problem is restricted on the space interval $[0,1]$ and the mesh size $ h = 2^{-7}$ is chosen for spatial discretization.
For any $\kappa > 0$ and periodic fractional order $\alpha(x)$, three homogeneous phases, namely, $u_1\equiv 1$, $u_{-1}\equiv -1$ and $u_0\equiv 0$, remain to be stationary solutions of \eqref{mh}. By analyzing the linear stability of the corresponding Jacobian operator, both $u_1$ and $u_{-1}$ are steady states of \eqref{AC}, while $u_0$ is a saddle point and its index increases as $\kappa$ or $\alpha$ decreases. Using $u_0$ as the parent state, we construct the solution landscape by downward search via saddle dynamics in Section \ref{sec:sd}.

\subsubsection{Constant-order case}

\begin{figure}[hbt]
	\centering
	\includegraphics[width=0.9\linewidth]{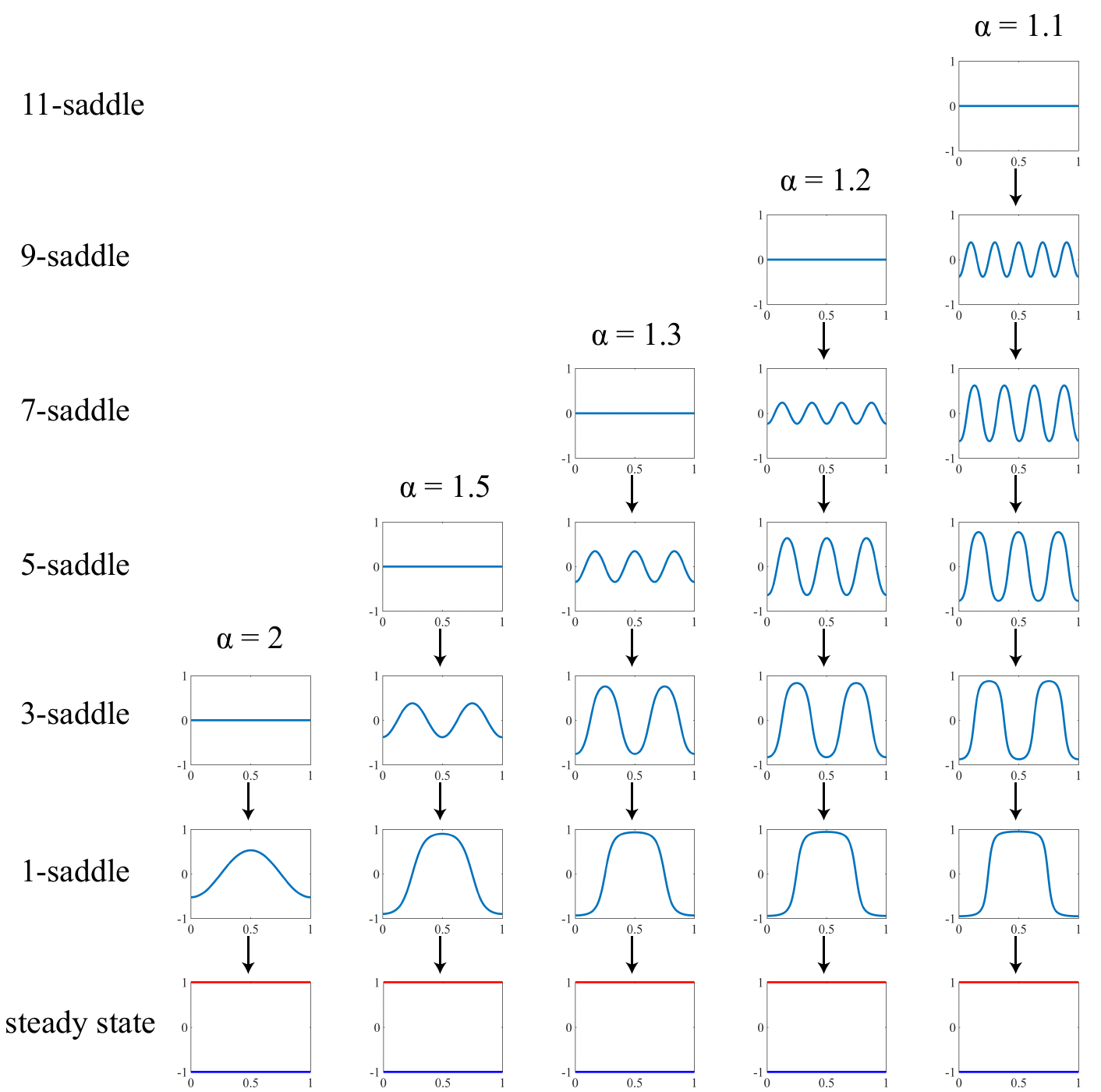}
	\caption{Solution landscapes of 1D space-fractional phase field model \eqref{mh} with different constant fractional order $\alpha$ at $\kappa = 0.02$.}
	\label{fig:constantorder1d}
\end{figure}
First, we fix the diffusion coefficient $\kappa = 0.02$ and vary the fractional order $\alpha$ to compute the corresponding solution landscapes in Figure \ref{fig:constantorder1d}. 
Each image in Figure \ref{fig:constantorder1d} denotes a stationary solution, and each column shows the solution landscape with different $\alpha$. Each solid arrow represents a saddle dynamics. 
Because of the periodic boundary condition, a translation of an inhomogeneous stationary solution is still a stationary solution, and only one phase is shown in the solution landscape. 
In addition, if $\hat{\bm u}$ is a stationary solution of \eqref{AC}, so is $-\hat{\bm u}$. 
As we can seen in Figure  \ref{fig:constantorder1d}, $u_1\equiv 1$ and $u_{-1}\equiv -1$ remain to be steady states with the changing of $\alpha$. The index-1 saddle points at different $\alpha$ show a single peak, but the interface becomes sharper as $\alpha$ is smaller. Furthermore, when $\alpha$ decreases, the index of the homogeneous state $u_0$ becomes larger, leading to the emergence of more high-index saddle points. 

The bifurcation graph in Figure \ref{fig:bifurcation}(A) is plotted to explain the jump of the index of $u_0$ and the appearance of new stationary solutions. With the order $\alpha$ decreasing, the non-homogeneous states with larger wave number will appear one by one through pitchfork bifurcations (denoted by dash lines), and then amplify their amplitudes.
We also present the index of $u_0$ against $\alpha$ at different $\kappa = 0.02, 0.015, 0.01$ in Figure  \eqref{fig:bifurcation}(B). The three curves show similar behaviors. Decreasing $\alpha$ leads to the increase of the index of $u_0$, and smaller $\kappa$ results in the higher-index of $u_0$.
Furthermore, a coarse graph of the index of $u_0$ on the parameter space $(\kappa, \alpha)$, shown in Figure \ref{fig:bifurcation}(C), demonstrates how $\alpha$ and $\kappa$ affect the index of $u_0$. It is worth mentioning that decreasing $\alpha$ and $\kappa$ has similar affects on the index of $u_0$. 

\begin{figure}[hbt]
	\centering
	\includegraphics[width=0.9\linewidth]{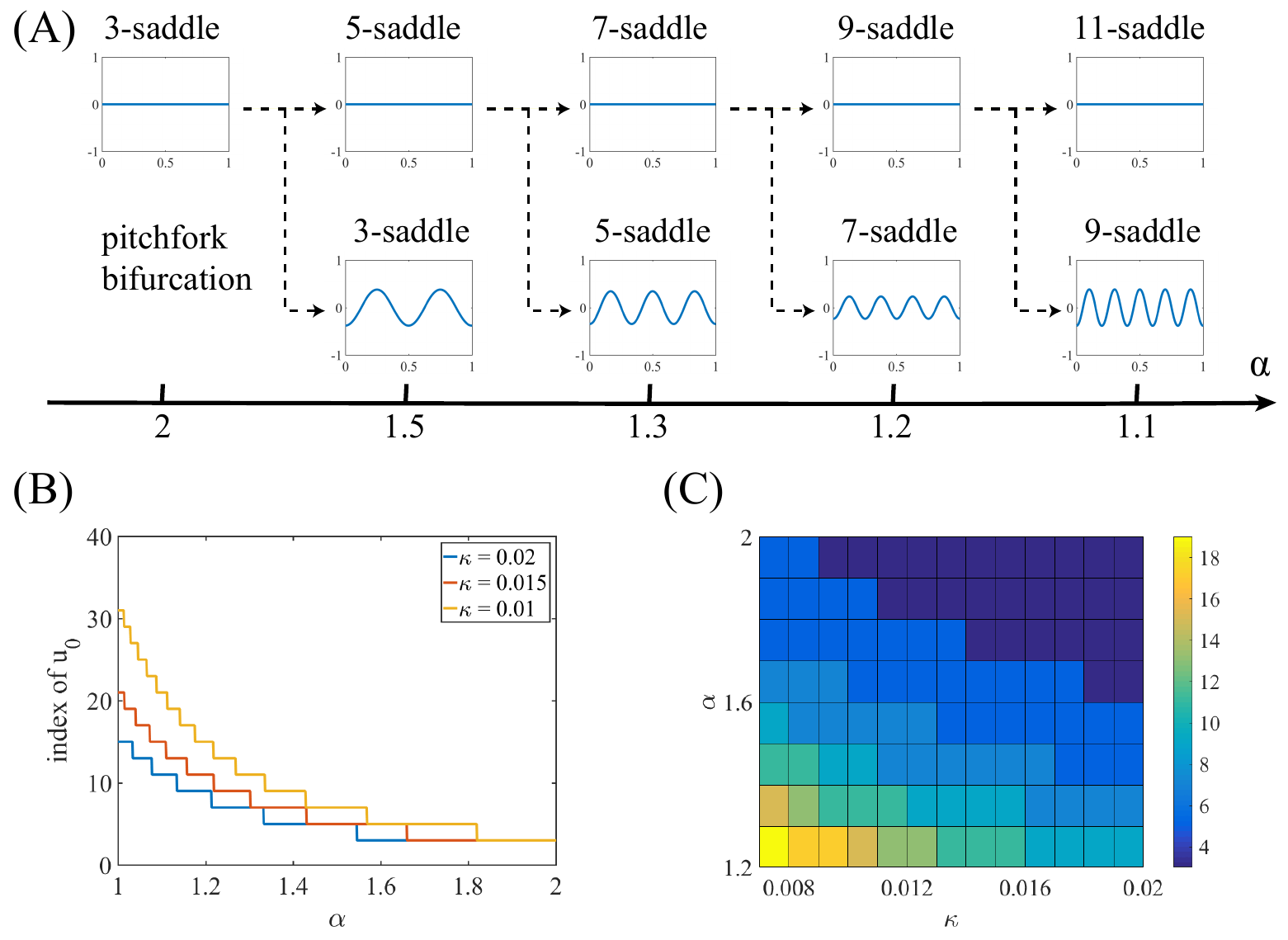}
	\caption{(A) Bifurcation graph with respect to the fractional order $\alpha$ at $\kappa = 0.02$. (B) Index of $u_0$ against $\alpha$ at $\kappa = 0.02, 0.015, 0.01$. (C) Index of $u_0$ on the space $(\kappa, \alpha)$. }
	\label{fig:bifurcation}
\end{figure}

Motivated by the finding in Figure \ref{fig:bifurcation}(C), we can choose suitable $\kappa$ in the integer-order phase field model \eqref{AC} by matching the index of $u_0$ at different $\alpha$ and compare their solution landscapes to reveal the relationship between the fractional Laplacian and the integer-order Laplacian. In Figure \ref{fig:alphakappa}(A), we compute the solution landscapes with different $\kappa$ at the integer-oder $\alpha = 2$. By matching the index order of $u_0$, the solution landscapes with different $\kappa$ are very similar to the solution landscapes with different $\alpha$ in Figure \ref{fig:constantorder1d}. 
To be more specific, we compare the index-$1$ saddle points with different $\kappa$ and $\alpha$ in Figure \ref{fig:alphakappa}(B) and (C), respectively. It shows similar interfaces between two phases and they becomes sharper with $\kappa$ or $\alpha$ decreasing. 
There is a subtle difference between the fractional order and the integer order. The maxima (minima) of saddle points with fractional order $\alpha$ could not reach $+1(-1)$, which is different from their integer-order analogues.
Nevertheless, comparison of these solution landscapes shows that the spectral fractional Laplacian operator exhibits similar properties as the integer-order Laplacian by choosing the suitable diffusion coefficient. In fact, because the spectral fractional laplacian is defined in terms of the eigenfunctions of the integer-order Laplacian, it may perform similar locality properties as integer-order differential operators.

\begin{figure}[hbt]
	\centering
	\includegraphics[width=0.85\linewidth]{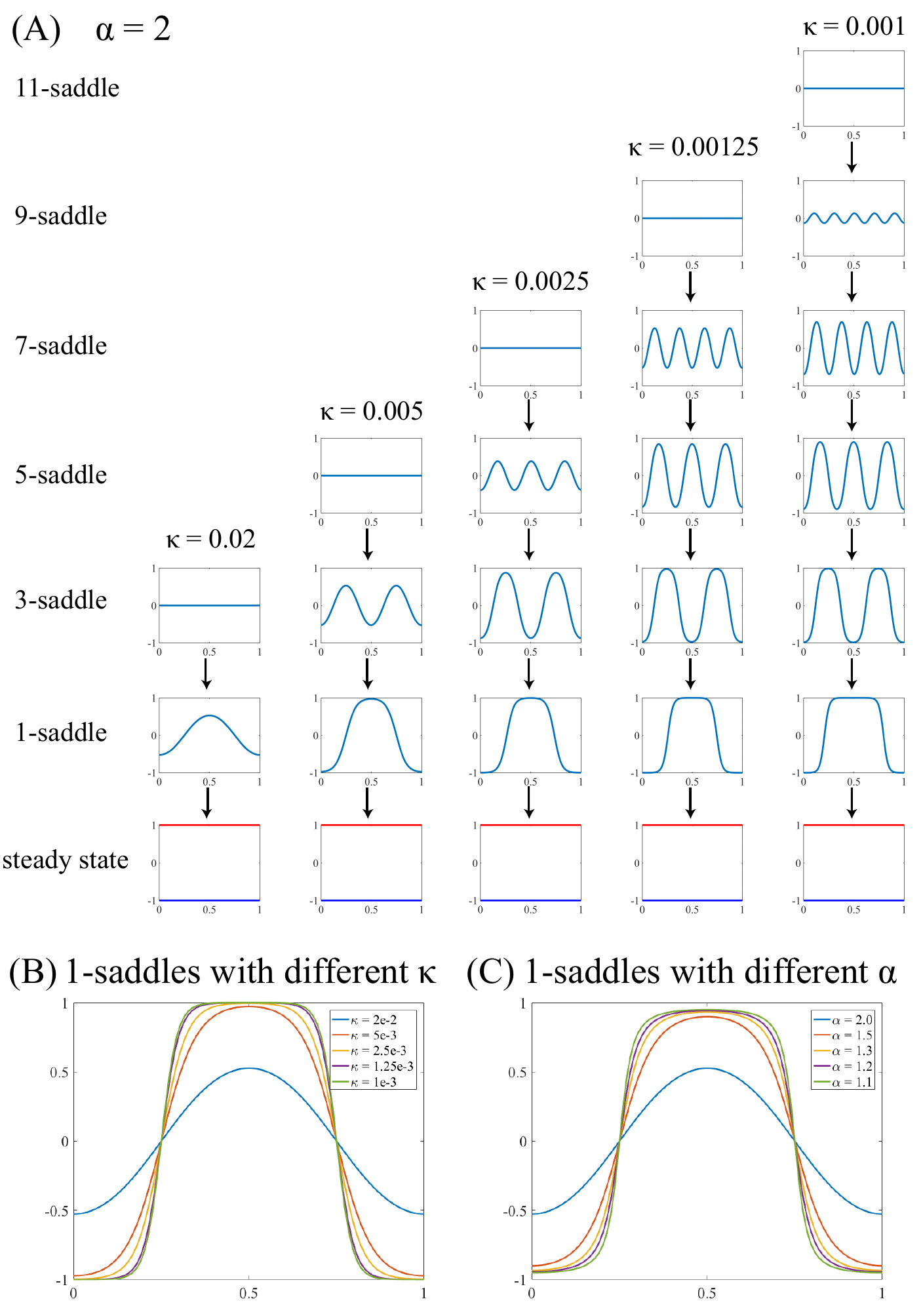}
	\caption{(A) Solution landscapes of 1D integer-order phase field model \eqref{AC} with different $\kappa$ at $\alpha = 2$. (B) The index-$1$ saddle points with different $\kappa$ at $\alpha = 2$. (C) The index-$1$ saddle points with different $\alpha$ at $\kappa = 2\times 10^{-2}$. }
	\label{fig:alphakappa}
\end{figure}

\subsubsection{Variable-order case}
We next compute the solution landscapes of the variable-order space-fractional phase field model. The following three trigonometric functions with different amplitudes
$$1.2 + 0.1\cos(2\pi x), 1.3 + 0.2\cos(2\pi x), 1.55 + 0.45 \cos(2 \pi x) $$ 
are used as the variable fractional order $\alpha(x)$ (cf. the top line of Figure  \ref{fig:variableorder1d}). Due to the loss of symmetry of the variable fractional order $\alpha(x)$, translations of saddle points are no longer stationary solutions as in the constant fractional order case in general. 

\begin{figure}[hbt]
	\centering
	\includegraphics[width=0.9\linewidth]{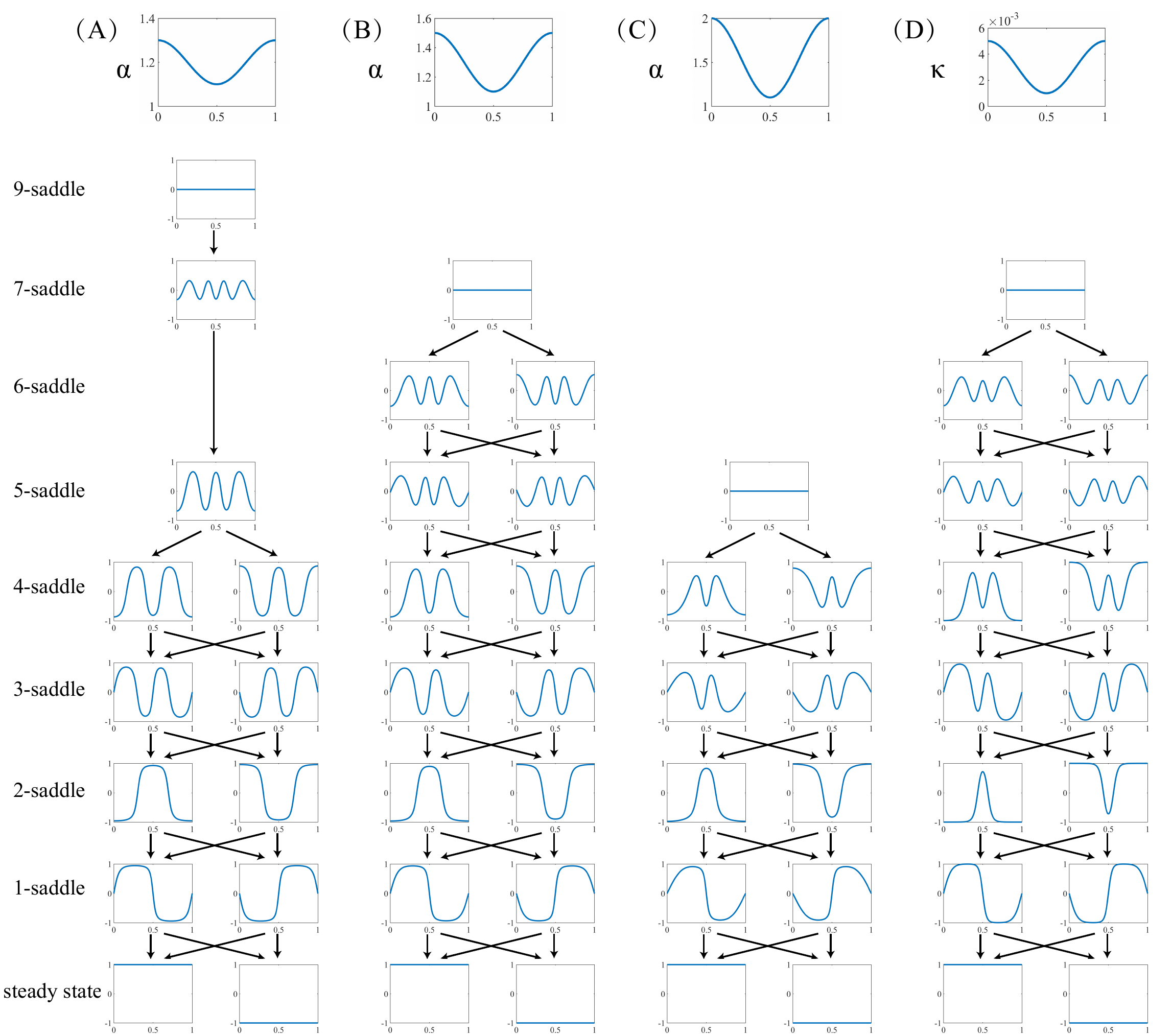}
	\caption{(A)-(C) Solution landscapes of 1D variable-order phase field model with different $\alpha(x)$. (D) Solution landscape of the integer-order phase field model \eqref{VC} with variable coefficient.}
	\label{fig:variableorder1d}
\end{figure}

In Figure \ref{fig:variableorder1d}(A-C), $\alpha(x)$ varies from the constant order $\alpha = 1.2$ gradually and thus the corresponding solution landscapes differ from that of the constant order $\alpha = 1.2$. More low-index saddle points emerge through pitchfork bifurcations, and the interfaces in the region of smaller $\alpha(x)$ are sharper than others. 
Furthermore, the index of $u_0$ decreases when the amplitude of $\alpha(x)$ becomes larger, resulting in some high-index saddle points disappear, which is consistent to the constant order case by increasing $\alpha$ (cf. Figure \ref{fig:constantorder1d}). 

To answer the question that whether it is possible to get similar solutions of variable-order fractional problems by varying the coefficients in their integer-order analogues, we compute the following variable-coefficient phase field model,
\begin{equation}\label{VC}
	\dot{u} = \kappa(x) \Delta u + u - u^3, 
\end{equation}
with 
$$\kappa(x) = 3\times 10^{-3} + 2\times 10^{-3}\cos(2\pi x).$$
This specific $\kappa(x)$ is chosen by matching the index of $u_0$ under the $\alpha(x)$ in Figure \ref{fig:variableorder1d}(B) for comparison. 
In Figure \ref{fig:variableorder1d}(D), we compute the solution landscape for the variable-coefficient case.
Comparing the solution landscapes in Figure \ref{fig:variableorder1d} (B) and (D), we find that all solutions of the variable-order fractional phase field model could be approximately reconstructed by adjusting the variable coefficient in the integer-order phase field model. 
In a similar manner, the solution landscapes in Figure \ref{fig:variableorder1d}(A) and (C) can be reconstructed from the model \eqref{VC} by choosing
$\kappa(x) = 1.25\times 10^{-3} + 2.5\times 10^{-4}\cos(2\pi x)$ and $  1.05\times 10^{-2} + 9.5\times 10^{-3}\cos(2\pi x), $
respectively. 
Numerical simulations show that, similar to the constant fractional order case, the variable-order spectral fractional Laplacian operator exhibits similar behaviors as the integer-order Laplacian with a suitable variable-coefficient. This again demonstrates that the variable-order spectral fractional Laplacian is close to a local differential operator from the view point of solution landscape.

\subsection{Two-dimensional problems}
To show the efficiency of numerical methods, we further test 2D space-fractional phase field model \eqref{AC} on a square domain $[0,1]^2$ with the mesh grid $N = 2^6$ in each direction for spatial discretization. 

\begin{figure}[hbt]
	\centering
	\includegraphics[width=0.9\linewidth]{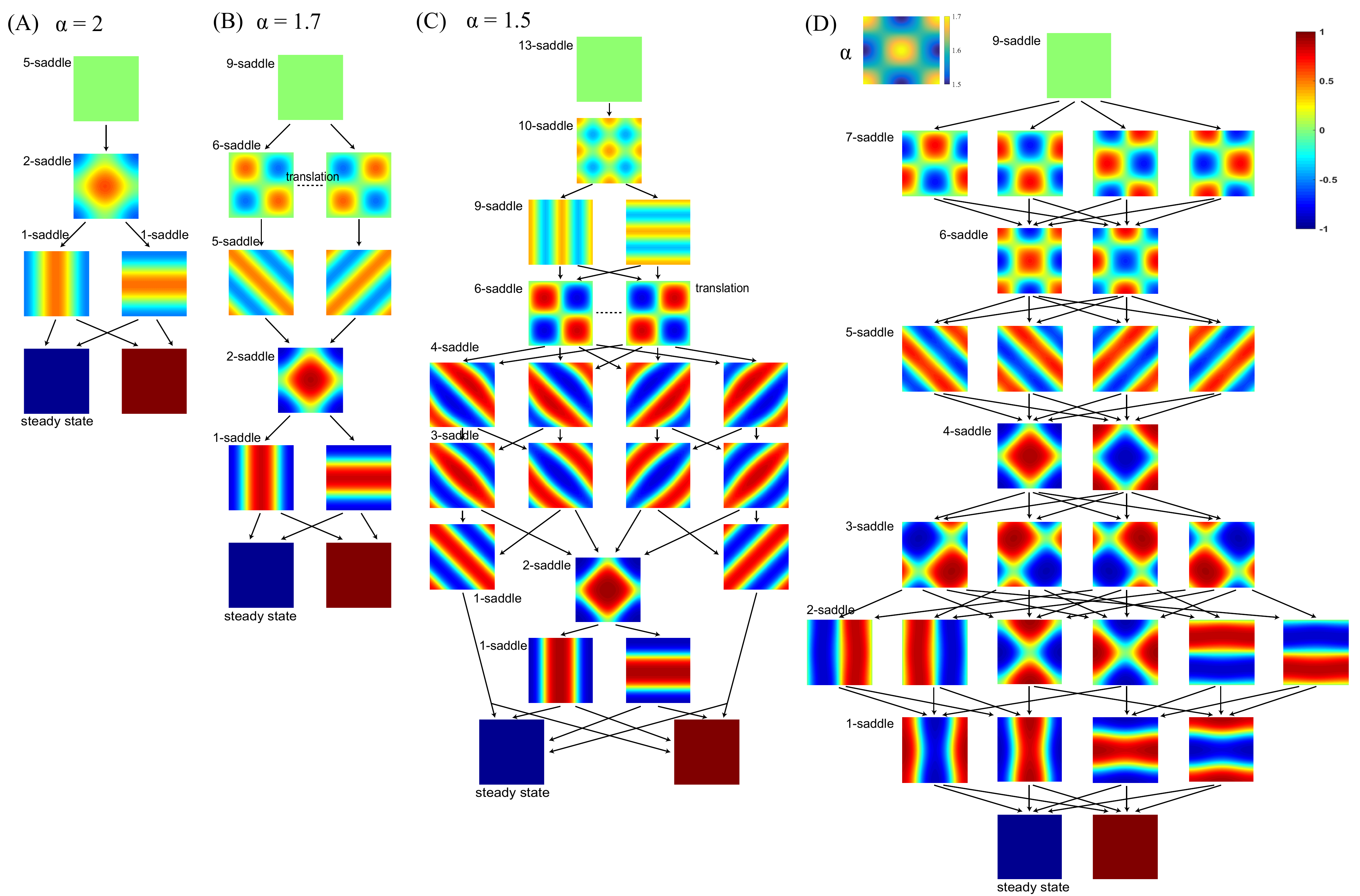}
	\caption{Solution landscapes of 2D space-fractional phase field model with the constant order $\alpha = 2$ (A), $\alpha = 1.7$ (B), $\alpha = 1.5$ (C), and the variable-order $\alpha(x,y) = 1.6 + 0.1 \cos(2 \pi x) \cos(2 \pi y)$ (D) .}
	\label{fig:sl2d}
\end{figure}

The solution landscapes for both constant-order and variable-order cases are shown in Figure \ref{fig:sl2d}. 
When $\alpha = 2$ in Figure \ref{fig:sl2d}(A), which is an integer-order case, the homogeneous state $u_0 \equiv 0$ is an index-$5$ saddle point of the system, and connects a 2-saddle square state, two 1-saddle stripe states and two homogeneous steady states $u_1\equiv 1$ and $u_{-1}\equiv -1$.
In Figure \ref{fig:sl2d}(B) and (C), when $\alpha$ decreases, the index of $u_0$ increases and various states, which are some mixtures of 1D solutions in some way, appear through pitchfork bifurcations. 
We need to point out that, similar stationary solutions in Figure \ref{fig:sl2d}(B) and (C) could be discovered with $\kappa = 0.01, 0.006$ at integer order $\alpha = 2$ \cite{YinSCM}, which again demonstrates that the parameters $\alpha$ and $\kappa$ have comparable affects on the solution landscapes of the space-fractional phase field model.

Furthermore, we choose the variable fractional order as
$$\alpha(x,y) = 1.6 + 0.1 \cos(2\pi x)\cos(2\pi y),$$
and the corresponding solution landscape is shown in Figure \ref{fig:sl2d}(D). More interesting saddle points are identified in the solution landscape of the variable-order phase field model.
The results indicate that, by designing suitable variable fractional orders, specific stationary solutions may be achieved for particular applications.

\section{Conclusions}
In this paper we develop and analyze a fast approximation algorithm for the variable-order space-fractional Laplacian, and implement it in saddle dynamics for efficient construction of the solution landscapes of nonlinear variable-order space-fractional phase field model. By approximating the variable-indexing Fourier modes of the fractional Laplacian by the series expansion method, we prove that the high-order accuracy of the Fourier mode approximation could be obtained using only $O(\ln M)$ terms in the expansion. We then construct the solution landscape of both constant- and variable-order nonlinear space-fractional evolution problems via saddle dynamics. Numerical experiments are performed to substantiate the accuracy and the efficiency of fast approximation algorithm and elucidate the global structures of the stationary solutions of nonlinear space-fractional problems. We further demonstrate the similarities and differences between the spectral fractional Laplacian problems and their integer-order analogues from the viewpoint of solution landscape. 
We conclude from the perspective of the solution landscapes that the stationary solutions of spectral fractional Laplacian problems can be reconfigured by choosing the suitable diffusion coefficients in corresponding integer-order problems, which indicates the locality of the spectral fractional Laplacian operator as the integer-order differential operators.

\section*{Acknowledgements}
This work was partially supported by the National Natural Science Foundation of China No.~12050002 and 21790340, the  International Postdoctoral Exchange Fellowship Program (Talent-Introduction Program) No.~YJ20 210019, and the China Postdoctoral Science Foundation No.~2021TQ0017.

\end{document}